\documentclass[12pt,eps]{article}
\usepackage{hadronic}
\usepackage{graphicx}
\usepackage{epstopdf}
\usepackage{amsmath}
\usepackage{amssymb}
\usepackage{amscd}
\begin{document}
\thispagestyle{empty}

\title {\bf  RINGS IN WHICH EVERY NILPOTENT IS CENTRAL}
\author{{\bf Burcu Ungor and Sait Halicioglu}\\
{\small Department of Mathematics, Ankara University, Turkey}\\
{\small $\texttt{bungor@science.ankara.edu.tr}~~~~~$}  {\small
$\texttt{halici@ankara.edu.tr}$}\\
{\bf Handan Kose}\\
{\small Department of Mathematics, Ahi Evran University, Turkey}\\
{\small $\texttt{hkose@ahievran.edu.tr}$}\\ and\\
{\bf Abdullah Harmanci}\\
{\small  Department of Mathematics, Hacettepe University,  Turkey}\\
{\small $\texttt{harmanci@hacettepe.edu.tr}$}}

\date{}
\maketitle

\newtheorem {thm}{Theorem}[section]
\newtheorem{lem}[thm]{Lemma}
\newtheorem{prop}[thm]{Proposition}
\newtheorem{cor}[thm]{Corollary}
\newtheorem{df}[thm]{Definition}
\newtheorem{nota}{Notation}
\newtheorem{note}[thm]{Remark}
\newtheorem{ex}[thm]{Example}
\newtheorem{exs}[thm]{Examples}
\newtheorem{rmk}[thm]{Remark}
\newtheorem{que}[thm]{Question}
\newenvironment{proof}{\par\noindent{\bf Proof. \,}}

\begin{abstract} In this paper, we introduce a class of rings in
which every nilpotent element is central. This class of rings
generalizes so-called reduced rings. A ring $R$ is called {\it
central reduced} if every nilpotent element of $R$ is central.
 For a ring $R$,  we prove that $R$ is central reduced if and only if
  $R[x_1,x_2,\ldots,x_n]$ is central reduced if and
only if  $R[[x_1,x_2,\ldots,x_n]]$ is central reduced if and only
if $R[x_1,x_1^{-1},x_2,x_2^{-1},\ldots,x_n,x_n^{-1}]$ is central
reduced. Moreover,  if $R$ is a central reduced ring, then the
trivial extension $T(R,R)$ is central  Armendariz.\\


\end{abstract}

\newpage
\section{Introduction}
Throughout this paper all rings are associative with identity
unless otherwise stated. A ring is {\it reduced} if it has no
nonzero nilpotent elements.  A ring $R$ is called
$semicommutative$ if for any $a,b \in R$, $ab=0 $ implies $aRb=0$.
Recently a generalization of semicommutative rings is given in
\cite{AOH}. A ring $R$ is called {\it central semicommutative}  if
for any $a,b\in R$, $ab=0$ implies  $arb$ is a central element of
$R$ for each $r\in R$. A ring $R$ is said to be $abelian$ if every
idempotent in $R$ is central. A ring $R$ is called {\it right
(left)
 principally quasi-Baer }\cite{BK}  if the right (left) annihilator of a
principal right (left) ideal of $R$ is generated by an idempotent.
Finally, a ring $R$ is called {\it right (left) principally
projective} if the right (left) annihilator of an element of $R$
is generated by an idempotent \cite{BKP}.

In this paper,  we introduce central reduced rings as a
generalization of reduced  rings. Clearly, reduced rings are
central reduced.  We supply some examples to show that all central
reduced rings need not be reduced.  Among others we prove that
central reduced rings are abelian and there exists an abelian ring
but not central  reduced. Therefore the class of central reduced
 rings lies strictly between  classes of reduced rings and
abelian rings. We prove that every central reduced ring is weakly
semicommutative, central semicommutative,  2-primal, abelian and
so directly finite, and a ring $R$ is central reduced if and only
if the  Dorroh extension of $R$ is central reduced. Moreover, it
is proven that if $R$ is a right principally projective ring, then
$R$ is central reduced  if and only if $R[x]/(x^n)$ is central
Armendariz,  where $n\geq 2~$ is a natural number and $(x^n)$ is
the ideal generated by $x^n$. It is shown that a ring $R$ is
central reduced if and only if the polynomial ring
$R[x_1,x_2,\ldots,x_n]$ is central reduced if and only if the
power series ring $R[[x_1,x_2,\ldots,x_n]]$ is central reduced if
and only if the Laurent polynomial ring
$R[x_1,x_1^{-1},x_2,x_2^{-1},\ldots,x_n,x_n^{-1}]$ is central
reduced. Finally we prove that if $R$ is a central reduced ring,
then the trivial extension $T(R,R)$ is central Armendariz.


Throughout this paper, $\Bbb{Z}$ denotes the ring of integers and
for a positive integer $n$, $\Bbb{Z}_n$ is the ring of integers
modulo $n$.

\section{Central Reduced Rings} In this section  we introduce a class of rings in
which every nilpotent element is central. We now give our main
definition.

\begin{df} A ring $R$ is called {\it central reduced } if every nilpotent element of $R$ is central.
\end{df}


 Commutative rings and  reduced rings are central
reduced. Every unit-central ring (i.e., every unit element of $R$
is central \cite{KMS}) is central reduced. One may suspect that
central reduced rings are reduced.  We now give an example to show
that central reduced rings need not be reduced.

\begin{ex}{\rm Let $S$ be a  commutative ring and $R=S[x]/(x^2)$. Then $R$ is
a commutative ring and so  it is central reduced. If $a=x+(x^2)
\in R$, then $a^2=0$. Therefore $R$ is not  a reduced ring.
}\end{ex}


Recall that a ring $R$ is {\it semiprime}  if $aRa=0$ implies
$a=0$ for $a \in R$.  Our next aim is to find conditions under
which a central reduced ring is reduced.

\begin{prop}\label{crred} If $R$ is a reduced ring, then $R$ is central reduced.
The converse holds if $R$ satisfies any of the following
conditions.
\begin{enumerate}
    \item[{\rm (1)}]  $R$ is a semiprime ring.
    \item[{\rm (2)}] $R$ is a right (left) principally projective ring.
     \item[{\rm (3)}] $R$ is a right (left) principally quasi-Baer ring.
\end{enumerate}
\end{prop}

\begin{proof} First statement is clear. Conversely, assume that $R$ is a central
reduced ring and $a\in R$ with $a^2=0$. Then $a$ is central. Now
consider the following cases.

(1) Let $R$ be a semiprime ring. Since $axa=0$ for all $x\in
    R$, it follows that $a=0$. Therefore $R$ is reduced.

(2) Let $R$ be a right principally projective ring. Then there
exists an idempotent $e\in R$ such that $r_R(a)=eR$. Thus
$a=ea=ae=0$, and so  $R$ is reduced. A similar proof may be given
for left principally projective rings.

(3) Same as the proof of (2). \hfill
$\Box$\par\bigskip\par\bigskip\end{proof}

\begin{cor} If $R$ is a central reduced ring, then the following
conditions are  equivalent.
\begin{enumerate}
    \item[{\rm (1)}] $R$ is a right  principally projective ring.
    \item[{\rm (2)}] $R$ is a left principally projective ring.
   \item[{\rm (3)}] $R$ is a right principally quasi-Baer ring.
    \item[{\rm (4)}] $R$ is a left principally quasi-Baer ring.
\end{enumerate}
\end{cor}

\begin{proof} It follows from Proposition \ref{crred} since in either case $R$ is
reduced. \hfill $\Box$\par\bigskip\end{proof}

Note that the homomorphic image of a central reduced  ring need
not be central reduced.  Consider the following example.

\begin{ex}{\rm Let $D$ be a division ring, $R=D[x,y]$  and $I=<x^2>$ where $xy\neq
yx$. Since $R$ is a domain, $R$ is central reduced. On the other
hand,  $x+I$ is a nilpotent element of $R/I$ but not central.
Hence $R/I$ is not central reduced. }
\end{ex}

We now determine under what conditions the homomorphic image of a
central reduced  ring is also central reduced.

\begin{prop} Let $R$ be a central reduced ring.  If $I$ is a nil ideal of $R$, then  $R/I$ is central
reduced. \end{prop}

\begin{proof} Let $a+I\in R/I$ with $(a+I)^n=\overline 0$ for  some positive integer
$n$. Then $a^n\in I$ and there exists a positive integer $m$ such
that $(a^n)^m=0$. Since $R$ is central reduced, $a$ is central.
Hence $ab-ba\in I$ for all $b\in R$ and so $a+I$ is central.
Therefore $R/I$ is central reduced. \hfill
$\Box$\par\bigskip\end{proof}

It is  well known that a ring is a domain if and only if it is
prime and reduced. In addition to this fact, we have the following
proposition when we deal with central case.

\begin{prop}\label{domain} Let $R$ be a ring. Then $R$ is a domain if and only if
$R$ is a prime and central reduced ring.
\end{prop}

\begin{proof} First assume  $R$ is a domain. It is clear that $R$ is prime and
reduced and so central reduced. Conversely, assume $R$ is a prime
and central reduced ring. Let $a, b\in R$ with $ab=0$. Then
$rab=0$ for all $r\in R$. Since $(bra)^2=0$, $bra$ and therefore
$bRa$ is contained in the center of $R$. Let $s \in R$ and
$asbrasb \in (asb)R(asb)$ for any $r\in R$. Hence
$asbrasb=abras^2b=0$, since $bra$ is central and $ab=0$. It proves
$(asb)R(asb)=0$. Being $R$ prime, we have $asb=0$ for all $s\in R$
and so $aRb=0$. By invoking the primeness of $R$ again we get
$a=0$ or $b=0$. Therefore $R$ is a domain. \hfill
$\Box$\par\bigskip\end{proof}

It is well known that every reduced ring is semicommutative. In
our case we have the following.

\begin{prop}\label{red-sem} Every central reduced ring is central semicommutative.
\end{prop}

\begin{proof} Let $R$ be a central reduced ring and $a, b\in R$ with
$ab=0$ and $r\in R$. Since $ab=0$, $ba$ is central. So
$barb=rbab=0$. Then $(arb)^2=0$. By hypothesis
 $arb$ is central.
\hfill $\Box$\par\bigskip\end{proof}

For any positive integer $n$ and a ring $R$,  let $T_n(R)$ denote
the $n\times n$ upper triangular matrix ring over the ring $R$ and
$R_n(R)$ denote the subring $\{ (a_{ij})\in T_n(R)\mid $ all
$a_{ii}$ 's are equal for  $i = 1,2,...,n\}$ of $T_n(R)$.

The following example shows that the converse of  Proposition
\ref{red-sem} may not be true in general.

\begin{ex}{\rm Let $F$ be a field. By \cite[Lemma 1.10]{AOH}  $R_3(F)$ is central semicommutative.
We prove that it is not central reduced. Consider  the nilpotent
element $A=\left[
\begin{array}{lll} 0&1&1\\ 0&0&0\\ 0&0&0 \end{array}\right]$. For  $B=\left[
\begin{array}{lll} 1&1&1\\ 0&1&1\\ 0&0&1 \end{array}\right]\in R_3(F)$,  $AB$ is not equal to
$BA$ and so  $R_3(F)$ is  not central reduced. }
\end{ex}

It is clear that prime semicommutative rings are reduced. For
central case we have the following result.
\begin{prop}\label{centralsemcom} Let $R$ be a  prime central semicommutative ring. Then $R$ is  reduced.
\end{prop}

\begin{proof} Let $a\in R$ with $a^n = 0$ for some positive integer
$n$. Since $R$ is central semicommutative, $a^{n-1}Ra$ is
contained in the center of $R$. Hence $(axa^{n-1})R(axa^{n-1})=0$
for all $x\in R$. By hypothesis $axa^{n-1}=0$ for all $x\in R$,
and so $aRa^{n-1}=0$. Thus $a=0$ or $a^{n-1}=0$. If $a=0$, then
the proof is completed. If $a^{n-1}=0$, then we also have $a=0$,
by using the similar technique as above. \hfill
$\Box$\par\bigskip\end{proof}

The next example shows that for a ring $R$ and an ideal $I$, if
$R/I$ is central reduced,  then $R$ need not be central reduced.

\begin{ex}{\rm Let $R=\left[ \begin{array}{ll} F&F\\
0&F\end{array}\right]$,  where $F$ is any field. Since $\left[ \begin{array}{ll} 0&1\\
0&0\end{array}\right]$ is a nilpotent but not central element of
$R$, $R$ is not central reduced. Now consider the ideal
$I=\left[ \begin{array}{ll} F&F\\
0&0\end{array}\right]$ of $R$. Then $R/I$ is central reduced
because of the commutativity property of $R/I$. }
\end{ex}

\begin{lem} Let $R$ be a prime ring. If $R/I$ is a central
reduced ring with a reduced ideal $I$, then R is  a reduced ring.
\end{lem}

\begin{proof}Let $R/I$ be a  central reduced ring. By Proposition \ref{red-sem},
 $R/I$  is central
semicommutative. To complete the proof we show that $R$ is central
semicommutative. Let $a, b\in R$ with $ab=0$. Since $bIa\subseteq
I$ and $(bIa)^{2}=0$, $bIa=0$. Therefore $((aRb)I)^{2}=0$ and so
$(aRb)I=0$. Since $R/I$ is  central semicommutative and
$(a+I)(b+I)=I$, $aRb+I\in C(R/I)$, that is, $arbr_{1}-r_{1}arb\in
I$ for all $r,r_{1}\in R$ and so $(arbr_{1}-r_{1}arb)^{2}\in
(arbr_{1}-r_{1}arb)I=0$ by $(aRb)I=0$. Then for all $r,r_{1}\in R$
we have $arbr_{1}=r_{1}arb$ and so $aRb$ is contained in the
center of $R$. Thus $R$ is central semicommutative. By Proposition
\ref{centralsemcom}, $R$ is reduced.
 \hfill $\Box$\par\bigskip\end{proof}

Recall that a ring R is called {\it weakly semicommutative}
\cite{LWL}, if for any $a, b\in R$, $ab = 0$ implies $arb$ is a
nilpotent element for each $r\in R$.

\begin{prop}  Let $R$ be a central reduced ring. Then $R$ is
weakly semicommutative.
\end{prop}

\begin{proof} Let $a, b\in R$ with $ab=0$. Since
$R$ is central reduced,  $ba$ is  central in $R$. Hence for each
$r\in R$, $(arb)^2=arbarb=ar^2bab=0$. Therefore $R$ is a weakly
semicommutative ring. \hfill $\Box$\par\bigskip\end{proof}

The following example shows that there is a weakly semicommutative
ring which is not central reduced.

\begin{ex}{\rm  In \cite{LWL}, it is proved that $R_5(R)$ is a weakly semicommutative
ring for a ring $R$. Now consider the nilpotent element $a=\left[
\begin{array}{lllll} 0&0&0&1&1\\
0&0&0&0&0\\
0&0&0&0&0\\
0&0&0&0&0\\
0&0&0&0&0
\end{array}\right]$ of $R_5(R)$. Since $ab\neq ba$ for $b=\left[
\begin{array}{lllll} 1&0&0&0&0\\
0&1&0&0&0\\
0&0&1&0&0\\
0&0&0&1&1\\
0&0&0&0&1
\end{array}\right]\in R_5(R)$, $R_5(R)$ is not central reduced.}
\end{ex}

Let $P(R)$ denote the prime radical and $N(R)$ the set of all
nilpotent elements of the ring $R$. The ring $R$ is called {\it
$2$-primal} if $P(R) = N(R)$ (See namely \cite{Hi} and
\cite{HJP}). In \cite[Theorem 1.5]{Sh} it is proved that every
semicommutative ring is $2$-primal. In this direction we prove

\begin{thm}\label{primal}  Every central reduced ring is
2-primal. The converse holds for semiprime rings.
\end{thm}

\begin{proof}   Let $R$ be a central reduced ring. It is obvious that
$P(R)\subseteq N(R)$. For the converse inclusion, let $a\in N(R)$
with $a^n = 0$ for some positive integer $n$. Then
$(RaR)^n=0\subseteq P(R)$, and so $RaR\subseteq P(R)$. Hence we
have $a\in P(R)$. Therefore $N(R)\subseteq P(R)$.  Conversely, let
$R$ be a semiprime and 2-primal ring. Then $P(R) = 0$ and so $N(R)
= 0$. Hence $R$ is reduced and so central reduced. This completes
the proof. \hfill $\Box$\par\bigskip\end{proof}


\begin{thm} Let $R$ be a ring. Then the following are equivalent.
\begin{enumerate}
    \item[{\rm (1)}] $R$ is central reduced.
   \item[{\rm (2)}] $R/P(R)$ is central reduced with $P(R)\subseteq
   C(R)$ where $C(R)$ is the center of $R$.
\end{enumerate}
\end{thm}
\begin{proof} (1) $\Rightarrow$ (2) Clear from Theorem
\ref{primal}.

(2) $\Rightarrow$ (1) Let $x\in R$ with $x^n = 0$ for some
positive integer $n$ and $\overline{R}$ denote the ring $R/P(R)$.
Since $\overline{R}$ is central reduced, $\overline{x}=x+P(R)$ is
central in $\overline{R}$. This implies that
$(\overline{R}\overline{x}\overline{R})^n=0\subseteq
P(\overline{R})$, and so
$\overline{R}\overline{x}\overline{R}\subseteq P(\overline{R})$.
Hence $\overline{x}\in P(\overline{R})=0$, thus $x\in P(R)$. By
hypothesis, $x$ is central in $R$. \hfill
$\Box$\par\bigskip\end{proof}

\begin{prop}\label{abel} Every central reduced ring is abelian.
\end{prop}

\begin{proof}  Let $R$ be a central reduced ring and $e^2=e\in R$.
 Then $xe-exe$ and $ex-exe$ are central for all
$x\in R$ since $(xe-exe)^2=0$ and $(ex-exe)^2=0$. Hence
$(xe-exe)e=0$ and $e(ex-exe)=0$ for all $x\in R$. So we have
$xe=ex$ for all $x\in R$. Therefore $R$ is abelian. \hfill
$\Box$\par\bigskip\end{proof}


Every abelian ring need not be central reduced, as the following
example shows.

\begin{ex}{\rm Consider the ring
$$R=\left \{ \left[ \begin{array}{ll} a&b\\ c&d \end{array} \right] : a\equiv d~(mod~2), b\equiv c\equiv 0~(mod~2)
\right\}$$ Since $\left[ \begin{array}{ll} 0&0\\ 0&0 \end{array}
\right]$ and $\left[ \begin{array}{ll} 1&0\\ 0&1 \end{array}
\right]$ are the only idempotents of $R$, $R$ is abelian. On the
other hand,  $\left[ \begin{array}{ll} 0&2\\ 0&0 \end{array}
\right]$ is a nilpotent element of $R$ but not central because
$\left[ \begin{array}{ll} 0&2\\ 0&0 \end{array} \right]\left[
\begin{array}{ll} 1&0\\ 0&3 \end{array} \right]\neq \left[ \begin{array}{ll} 1&0\\
0&3 \end{array} \right]\left[ \begin{array}{ll} 0&2\\ 0&0
\end{array} \right]$. Hence $R$ is not central reduced.}
\end{ex}

Recall that a ring $R$ is called {\it directly finite}  whenever
$a, b\in R$, $ab=1$ implies $ba=1$. Then we have the following.

\begin{cor} If $R$ is a central reduced ring, then $R$ is
directly finite.
\end{cor}

\begin{proof} Clear from Proposition \ref{abel} since every abelian ring is
directly finite.\hfill $\Box$\par\bigskip\end{proof}

A ring $R$ is called {\it nil clean} \cite{HC} if there exist an
idempotent $e$ and a nilpotent $b$ in $R$ such that $a=e+b$. We
can give a relation between central reduced and commutative rings
by using nil clean rings.

\begin{prop} Let $R$ be a central reduced ring. If $R$ is nil
clean, then it is commutative.
\end{prop}
\begin{proof} Let $a\in R$. Since $R$ is nil clean, that is,   an idempotent $e$ and
a nilpotent $b$ exist in $R$ such that $a=e+b$. By hypothesis, $e$
and $b$ are central, and so $a$ is central. \hfill
$\Box$\par\bigskip\end{proof}

 Let $I$ be an index set and $\{R_i\}_{i\in I}$ be a class of rings.
Then $R_i$ is central reduced for all $i\in I$ if and only if
$\prod \limits_{i\in I} R_i$ is central reduced. Then the next
result is an immediate consequence of Proposition \ref{abel}.

\begin{cor} Let $R$ be a ring. Then the following are equivalent.
\begin{enumerate}
    \item $R$ is central reduced.
    \item $R$ is abelian and for any idempotent $e\in R$, $eR$ and
    $(1-e)R$ are central reduced.
    \item There is a central idempotent $e\in R$ with $eR$ and
    $(1-e)R$ are central reduced.
\end{enumerate}
\end{cor}

Recall that a ring $R$ is said to be \emph{regular} if for any $a
\in R$ there exists $b \in R$ with $a = aba$, while a ring $R$ is
called \emph{strongly regular} if for any $a \in R$ there exists
$b \in R$ such that $a = a^2b$.

Now we give some relations between reduced, central reduced,
regular, strongly regular and abelian rings. Also following
theorem provides some conditions for the converses of Proposition
\ref{crred} and Proposition \ref{abel}(1).

\begin{thm} Let $R$ be a ring. Then the following are equivalent.
\begin{enumerate}
    \item[{\rm (1)}] $R$ is strongly regular.
   \item[{\rm (2)}] All $R$-modules are flat and $R$ is central reduced.
    \item[{\rm (3)}] All cyclic $R$-modules are flat and $R$ is central reduced.
    \item[{\rm (4)}] $R$ is regular and central reduced.
    \item[{\rm (5)}] $R$ is regular and  reduced.
    \item[{\rm (6)}] $R$ is regular and abelian.
    \end{enumerate}
\end{thm}

\begin{proof} (1) $\Rightarrow$ (2) Since $R$ is strongly regular,
$R$ is central reduced. On the other hand,  $R$ is regular and so
every $R$-module is flat.

(2) $\Rightarrow$ (3) and (3) $\Rightarrow$ (4) Obvious.

(4) $\Rightarrow$ (5) Let $a\in R$ with $a^2=0$. By hypothesis
there exists $b\in R$ such that $a=aba$. Since $ab$ is an
idempotent, we have $a=a^2b=0$ by Proposition \ref{abel}(1). Hence
$R$ is reduced.

(5) $\Rightarrow$ (6) Clear.

(6) $\Rightarrow$ (1) Let $a\in R$. By hypothesis, there exists
$b\in R$ such that $a=aba$. Since $ab$ is an idempotent, $ab$ is
central. Hence $a=a^2b$ and therefore $R$ is strongly regular.
\hfill $\Box$\par\bigskip\end{proof}

Let $S$ denote a multiplicatively closed subset of a ring $R$
consisting of central regular elements. Let $S^{-1}R$ be the
localization of $R$ at $S$. Then we have the following lemma.

\begin{lem}\label{local} A ring $R$ is central reduced if and only if
$S^{-1}R$ is central reduced.
\end{lem}

\begin{proof} It is routine.
\hfill $\Box$\par\bigskip\end{proof}

\begin{thm}\label{poly} Let $R$ be a ring. Then the following are
equivalent.\\
{\rm(1)} $R$ is central reduced.\\
{\rm (2)}  $R[x_1,x_2,\ldots,x_n]$ is central reduced.\\
{\rm (3)} $R[[x_1,x_2,\ldots,x_n]]$ is central reduced.\\
{\rm (4)}  $R[x_1,x_1^{-1},x_2,x_2^{-1},\ldots,x_n,x_n^{-1}]$ is
central reduced.
\end{thm}

\begin{proof} The equivalencies of (1), (2) and (3) are clear by
showing that  $R$ is central reduced if and only if $R[x]$ is
central reduced. One way is clear. So assume that $R$ is central
reduced. Let $f(x)=a_0+a_1x+\ldots+a_nx^n \in R[x]$ be nilpotent.
 To complete the proof it is enough to show  $a_0$, $a_1$,
$\ldots, a_n$ are central. If $f(x)^2=0$, then we have

\[
\begin{array}{llll}
 a_0 ^2&=0&\hspace{0.8in}(1)\\
a_0a_1+a_1a_0&=0&\hspace{0.8in}(2)\\
a_0a_2+a_1 ^2 +a_2a_0&=0&\hspace{0.8in}(3) \\
\cdots
\end{array}
\]

\noindent Then $a_{0}$ is central and by (2) we have $2a_0a_1=0$.
(3) implies $2a_0a_2+a_1 ^2=0$. The latter  gives $2a_0a_2=-a_1
^2$. Hence $a_1 ^4=0$. So $a_1$ is central.  We may continue in
this way to obtain $a_2$,$\ldots, a_n$ central. Now assume
$f(x)^3=0$. Then {\small

\[
\begin{array}{rlll}
 a_0 ^3&=0&\hspace{0.8in}(1)\\
a_0 ^2a_1+a_0a_1a_0+a_1a_0 ^2&=0&\hspace{0.8in}(2)\\
a_0 ^2a_2+a_0 a_1 ^2+a_0a_2a_0+a_1a_0a_1+a_1 ^2 a_0+a_2a_0 ^2&=0&\hspace{0.8in}(3) \\
a_0 ^2a_3+a_0 a_1a_2+a_0a_2a_1+a_0a_3a_0+a_1a_0a_2+a_1 ^3+ \\
a_1a_2a_0+a_2a_0a_1+a_2a_1a_0+a_3a_0 ^2&=0&\hspace{0.8in}(4) \\
\cdots
\end{array}
\]
}

\noindent (1) implies $a_0$ is central. Hence from (4) we have
$3a_3a_0 ^2+3a_2a_0a_1+3a_1a_2a_0+a_1^3 =0$. Since $a_0^3=0$ and
$a_0$ is central, $a_1^9=0$ and so $a_1$ is central. Continuing on
this way we get $a_2$,$\ldots, a_n$ central. Similarly if
$f(x)^m=0$, we may prove that all coefficients of  $f(x)$ are
central.

 \noindent (2) $\Rightarrow$ (4) Let $S=\{1,x,x^2,\ldots \}$ be a
 multiplicatively closed subset of $R[x]$. By Lemma \ref{local},
 $R[x,x^{-1}]=S^{-1}R[x]$ is central reduced.

\noindent (4) $\Rightarrow$ (2) Clear.
 \hfill $\Box$\par\bigskip\end{proof}

The following result can be easily obtain from Theorem \ref{poly}.

\begin{cor} Let $R$ be a ring and $G$  a finitely generated free
abelian group. Then the following are equivalent.
\begin{enumerate}
    \item[{\rm (1)}] $R$ is central reduced.
    \item[{\rm (2)}] $RG$ is central reduced.
\end{enumerate}
\end{cor}

It is known that reduced rings are nonsingular  and commutative
nonsingular rings are reduced. But in our case  there is no
relation between nonsingular and central reduced rings, as the
following examples show.

\begin{exs} {\rm (1)  The ring $\Bbb Z_4$ is central reduced but not
nonsingular.

 (2)  The ring of $2\times 2$ matrices over a field is
left and
    right nonsingular. On the other hand,  $\left [ \begin{array}{ll} 0&1\\ 0&0 \end{array} \right
    ]$ is a nilpotent element but not central. Thus this ring is
    not central reduced.}
\end{exs}

The {\it Dorroh extension} $D(R, \Bbb Z)=\{(r, n) : r\in R, n\in
\Bbb Z\}$ of a ring $R$ is a ring with operations
$(r_1,n_1)+(r_2,n_2)=  (r_1+r_2,n_1+n_2)$ and
$(r_1,n_1)(r_2,n_2)=(r_1r_2+n_1r_2+n_2r_1,n_1n_2)$. Obviously $R$
is isomorphic to the ideal $\{(r,0) : r\in R\}$ of $D(R, \Bbb Z)$.
Then we obtain  the following.

\begin{prop} A ring $R$ is central reduced if and only if the
 Dorroh extension $D(R, \Bbb Z)$ of $R$ is central reduced.
\end{prop}

\begin{proof}  Let $R$ be a central reduced ring and $(r,n)\in D(R, \Bbb
Z)$ with $(r,n)^m=0$ for  some positive integer $m$. Since
$n^m=0$, it follows that $n=0$ and so $r^m=0$. By hypothesis $r$
is central. Then $(r,n)(s,a)=(s,a)(r,n)$ for any $(s,a)\in D(R,
\Bbb Z)$. Therefore $D(R, \Bbb Z)$ is central reduced. The
converse is clear. \hfill $\Box$\par\bigskip\end{proof}

Let $R$ be a ring and $M$  an $(R,R)$-bimodule. Recall that the
{\it trivial extension} of $R$ by $M$ is defined to be ring
$T(R,M)=R\oplus M$ with the usual addition and the multiplication
$(r_1,m_1)(r_2,m_2)=(r_1r_2, r_1m_2+m_1r_2)$. This ring is
isomorphic to the ring $\left\{ \left[ \begin{array}{ll} r&m\\
0&r
\end{array} \right] : r\in R,\, m\in M \right\}$ with the usual
matrix operations and isomorphic to $R[x]/(x^2)$, where $(x^2)$ is
the ideal generated by $x^2$. The trivial extension of $R$ by $M$
need not be a central reduced ring, as the following example
shows.

\begin{ex} {\rm Let $\Bbb H$ be the division ring of quaternions over the real
numbers. Then $\Bbb H$ is a reduced ring but not commutative.
Consider the nilpotent element $\left[ \begin{array}{ll} 0&i\\
0&0\end{array} \right] $ of $T(\Bbb H,\Bbb H)$. Since  $\left[
\begin{array}{ll} 0&i\\0&0\end{array} \right]\left[
\begin{array}{ll}j&0\\0&j\end{array} \right] \neq\left[
\begin{array}{ll}j&0\\0&j
\end{array} \right]\left[ \begin{array}{ll}0&i\\0&0\end{array} \right]$, $T(\Bbb H,\Bbb
H)$ is not central reduced.}
\end{ex}

It can be easily shown that for a positive integer $n \geq 2$,
$M_n(R)$ and $T_n(R)$ can not be central reduced even if $R$ is
commutative. But we have the following result when we deal with
$T(R,R)$.

\begin{prop} Let $R$ be a ring. Then the following are equivalent.
\begin{enumerate}
    \item[{\rm (1)}] $R$ is commutative.
   \item[{\rm (2)}] $T(R,R)$ is central reduced.
\end{enumerate}
\end{prop}
\begin{proof} (1) $\Rightarrow$ (2) By hypothesis, $T(R,R)$ is
commutative, and so it is central reduced.

(2) $\Rightarrow$ (1) Let $x, y\in R$. Since $\left[ \begin{array}{ll} 0&x\\
0&0\end{array} \right]\in T(R,R)$ is nilpotent, it commutes with $\left[ \begin{array}{ll} y&0\\
0&y\end{array} \right]\in T(R,R)$. Hence we have $xy=yx$. \hfill
$\Box$\par\bigskip\end{proof}

 Let $R$ be a ring and $f(x)=\sum \limits_{i=0}^n a_ix^i$,
$g(x)=\sum \limits_{j=0}^s b_jx^j \in R[x]$. Rege and Chhawchharia
\cite{RC} introduce the notion of an Armendariz ring, that is,
$f(x)g(x)=0$ implies $a_ib_j=0$ for all $i$ and  $j$. The name of
the ring was given due to Armendariz who proved that reduced rings
satisfied this condition \cite{AE}. The interest of this notion
lies in its natural and useful role in understanding the relation
between the annihilators of the ring R and the annihilators of the
polynomial ring $R[x]$. So far,  Armendariz rings are generalized
in different ways (see namely, \cite{HKKwak}, \cite{LZh}). In
particular, a ring $R$ is called {\it linear Armendariz}
\cite{LW}, if the product of two linear polynomials in $R[x]$ is
zero, then each product of their coefficients is zero.  A ring $R$
is called {\it central linear Armendariz} \cite{AHH},  if the
product of two linear polynomials in $R[x]$ is zero, then each
product of their coefficients is central. According to Harmanci et
al. \cite{AGHH}, a ring $R$ is called {\it central Armendariz} if
 $f(x)g(x)=0$ implies that $a_ib_j$ is central element of $R$ for all $i$ and $j$.
A ring $R$ is called {\it weak Armendariz} \cite{LZh}  if
$f(x)g(x) = 0$, then $a_i b_j$ is a nilpotent element of $R$ for
each $i$ and $j$, while a ring $R$ is called {\it nil-Armendariz}
\cite{An} if $f(x)g(x)$ has nilpotent coefficients, then $a_ib_j$
is nilpotent for $0\leq i\leq n$, $0\leq j\leq s$. Clearly every
nil-Armendariz ring is weak Armendariz. In \cite[Theorem 5]{AC1},
Anderson and Camillo proved that for a ring $R$ and \linebreak $n
\geq 2$ a natural number, $R[x]/(x^n)$ is Armendariz if and only
if R is reduced.   For central reduced rings, we have

\begin{thm}\label{thm6} Let $R$ be a right principally  projective ring and $n\geq 2$ a natural number.
Then $R$ is central reduced  if and only if  $R[x]/(x^n)$ is
central Armendariz.
\end{thm}

\begin{proof} Suppose $R$ is a central reduced  ring. By Proposition
\ref{crred},  $R$ is a reduced ring. From \cite[Theorem 5]{AC1},
$R[x]/(x^n)$ is Armendariz and so central Armendariz. Conversely,
assume that $R[x]/(x^n)$ is central Armendariz. By hypothesis and
\cite[Theorem 2.5]{AGHH},   $R[x]/(x^n)$ is  Armendariz. It
follows from \cite[Theorem 5]{AC1} that $R$ is reduced and so
central reduced. \hfill $\Box$\par\bigskip\end{proof}

Central reduced rings allow us to get the following result.

\begin{thm}\label{burcu} Let $R$ be a central reduced ring. Then  the
followings hold.
\begin{enumerate}
    \item[{\rm (1)}] $R$ is nil-Armendariz.
    \item[{\rm (2)}] $R$  is weak Armendariz.
    \item[{\rm (3)}] $R$  is central Armendariz.
  \end{enumerate}
\end{thm}

\begin{proof} If $R$ is central reduced, then it is
2-primal  by Theorem \ref{primal} and so  $N(R)$  is an ideal of
$R$. Proposition 2.1 in \cite{An} states that in a ring in which
the set of all nilpotent elements forms an ideal, then the ring is
nil-Armendariz. Therefore $R$ is weak Armendariz and central
Armendariz. \hfill $\Box$\par\bigskip\end{proof}

\begin{cor} If $R$ is a central reduced ring, then  $R[x]/(x^n)$ is
nil-Armendariz.
\end{cor}

\begin{proof} If $R$ is central reduced, then it is nil-Armendariz
by Theorem \ref{burcu}.  From \cite[Proposition 4.1]{An},
 $R[x]/(x^n)$ is  nil-Armendariz.
 \hfill $\Box$\par\bigskip\end{proof}

 We now give a useful lemma without proof to show that the trivial extension of
 a central reduced ring is central Armendariz.

\begin{lem}\label{attr} The following hold for a ring $R$ with
$a,b \in R$.
\begin{enumerate}
\item[{\rm (1)}] The sum of central nilpotent elements of $R$ is nilpotent.
\item[{\rm (2)}] If  $b$ is central nilpotent, then $ba$ and $ab$ are  nilpotent.
\item[{\rm (3)}] If  $aba$ is central nilpotent, then $ab$ and $ba$ are nilpotent.
\end{enumerate}
\end{lem}

Note that if $R$ is a reduced ring, by \cite[Proposition 2.5]{RC}
trivial extension $T(R,R)$ is Armendariz and so it is central
Armendariz. In \cite[ Lemma 2.18] {AHH}, it is proved that for a
central reduced ring $R$,  the trivial extension $T(R,R)$ is
central linear Armendariz. Here we extend this result to central
Armendariz rings. Note that in proving Theorem \ref{trr} we use
the results in Lemma \ref{attr} without mention.

\begin{thm}\label{trr}  If $R$ is a central reduced ring, then the trivial extension $T(R,R)$ is central  Armendariz.

\end{thm}

\begin{proof} Let\\ $f(x)=\left [
\begin{array}{cc}
a_0 & a'_0\\
0 & a_0 \end{array}\right]+ \left [
\begin{array}{cc}
a_1 & a'_1\\
0 & a_1 \end{array}\right]x +\ldots + \left [
\begin{array}{cc}
a_n & a'_n\\
0 & a_n \end{array}\right]x^n =\left [
\begin{array}{cc}
f_{1}(x) & f_{2}(x)\\
0 & f_{1}(x) \end{array}\right]$,\\ $g(x)=\left [
\begin{array}{cc}
b_0 & b'_0\\
0 & b_0 \end{array}\right]+ \left [
\begin{array}{cc}
b_1 &b'_1\\
0 & b_1 \end{array}\right]x + \ldots + \left [
\begin{array}{cc}
b_t & b'_t\\
0 & b_t \end{array}\right]x^t = \left [
\begin{array}{cc}
g_{1}(x) & g_{2}(x)\\
0 & g_{1}(x) \end{array}\right]$ be elements of $T(R,R)[x]$, where
$f_1(x) = a_0 + a_1x + ... + a_nx^n$, $f_2(x) = a'_0 + a'_1x + ...
+ a'_nx^n$, $g_1(x) = b_0 + b_1x + ... + b_tx^t$, $g_2(x) = b'_0 +
b'_1x + ... + b'_tx^t$. Assume that $f(x)g(x)=0$. We have
 $$f(x)g(x)=\left [
\begin{array}{cc}
 f_{1}(x)g_{1}(x)& f_{1}(x)g_{2}(x)+f_{2}(x)g_{1}(x)\\
0 & f_{1}(x)g_{1}(x)
\end{array}\right]=0.$$  Hence $f_{1}(x)g_{1}(x)=0$ and
$f_{1}(x)g_{2}(x)+f_{2}(x)g_{1}(x)=0$. By Theorem \ref{burcu}, $R$
is both nil-Armendariz and central Armendariz. Hence $a_ib_j$ and
$b_ja_i$ are central nilpotent in $R$  for all $i,j$. As for the
coefficients of $f_{1}(x)g_{2}(x)+f_{2}(x)g_{1}(x)=0$, we have
\[
\begin{array}{llll}
 a_0b'_0 + a'_0b_0&=0&\hspace{0.8in}(1)\\
a_0b'_1+a_1b'_0 + a'_0b_1 + a'_1b_0&=0&\hspace{0.8in}(2)\\
a_0b'_2+a_1b'_1 + a_2b'_0 + a'_0b_2 + a'_1b_1 +
a'_2b_0&=0&\hspace{0.8in}(3) \\
\cdots
\end{array}
\]
\noindent Multiplying (1) from the right by $a_0$, we have
$a_0b'_0a_0 + a'_0b_0a_0 = 0$. Since $b_0a_0$ is central
nilpotent, $a_0b'_0a_0$ is central nilpotent, so $a_0b'_0$ and
$b'_0a_0$ are central nilpotent. From (1) $a'_0b_0$ and $b_0a'_0$
are central nilpotent.

Multiplying (2) from the right by $a_0$, we have
$a_0b'_1a_0+a_1b'_0a_0 + a'_0b_1a_0 + a'_1b_0a_0 = 0$. Since
$b'_0a_0$, $b_1a_0$ and $b_0a_0$ are central nilpotent,
$a_0b'_1a_0$ and so $a_0b'_1$ and $b'_1a_0$ are central nilpotent.
Now multiplying (2) from the left by $b_0$ and using central
nilpotency of $b_0a_0$, $b_0a_1$ and $b_0a'_0$,  we have
$b_0a'_1b_0$ is central nilpotent and so $b_0a'_1$ and $a'_1b_0$
are central nilpotent. Multiplying (2) from the right by $a_1$ and
using $a_0b'_1a_1$, $a'_0b_1a_1$ and $a'_1b_0a_1$ are central
nilpotent we have $a_1b'_0a_1$ and so $a_1b'_0$ and $b'_0a_1$ are
central nilpotent. The remaining term of (2) $a'_0b_1$ and
$b_1a'_0$ are central nilpotent.

Similarly, multiplying (3) from the right by $a_0$ and using
$b'_1a_0$, $b'_0a_0$, $b_2a_0$, $b_1a_0$ and $b_0a_0$ are central
nilpotent, we have $a_0b'_2a_0$, therefore $b'_2a_0$ and $a_0b'_2$
are central nilpotent. Multiplying (3) from the left by $b_0$ and
using $b_0a_0$, $b_0a_1$, $b_0a_2$, $b_0a'_0$ and $b_0a'_1$ are
central nilpotent, we have $b_0a'_2b_0$, therefore $b_0a'_2$ and
$a'_2b_0$ are central nilpotent. Multiplying (3) by $a_1$ from
right we have $a_0b'_2a_1$, $a_2b'_0a_1$, $a'_0b_2a_1$,
$a'_1b_1a_1$, $a'_2b_0a_1$ are central nilpotent. Hence
$a_1b'_1a_1$ and therefore $a_1b'_1$ and $b'_1a_1$ are central
nilpotent. Similarly, multiplying (3) by $b_1$ from the left we
have $b_1a_0b'_2$, $b_1a_1b'_1$, $b_1a_2b'_0$, $b_1a'_0b_2$ and
$b_1a'_2b_0$ are central nilpotent, therefore $b_1a'_1b_1$ and so
$b_1a'_1$ and $a'_1b_1$ are central nilpotent. Multiplying (3)
from the right by $a_2$ and using $a_0b'_2a_2$, $a_1b'_1a_2$,
$a'_0b_2a_2$, $a'_1b_1a_2$, $a'_2b_0a_2$ are central nilpotent,
then $a_2b'_0a_2$ and so $a_2b'_0$ and $b'_0a_2$ are central
nilpotent. $a'_0b_2$ in (3) is central nilpotent since it is a sum
of  central nilpotent elements, so is $b_2a'_2$. Thus all terms of
(3) are central nilpotent.

To complete the proof,  we use induction on $i + j$ for $i+j\leq
n+t$. Assume that the claim is true for all $i + j - 1$ where $i +
j\leq n + t$, that is for all $k$ and $l$ with $k+l\leq i+j-1$,
$a_kb'_l$, $b'_la_k$, $a'_kb_l$, $b_la'_k$ are central nilpotent.
Consider the $(i + j)$-th equation

\noindent $a_0b'_{i+j}+a_1b'_{i+j-1} +
a_2b'_{i+j-2}+...+a_{i+j-1}b'_1+a_{i+j}b'_0 +$ $a'_0b_{i+j} +
a'_1b_{i+j-1}+...+ a'_{i+j-2}b_2+a'_{i+j-1}b_1+a'_{i+j}b_0= 0$.

\noindent In order to complete the proof by induction,  we have to
show that all terms  are central nilpotent in $(i+j)$. We proceed
as preceding. Multiplying $(i + j)$ by $a_0$ from the right we
have $a_0b'_{i+j}a_0+a_1b'_{i+j-1}a_0 +
a_2b'_{i+j-2}a_0+...+a_{i+j-1}b'_1a_0+a_{i+j}b'_0a_0+$
$a'_0b_{i+j}a_0 + a'_1b_{i+j-1}a_0+...+
a'_{i+j-2}b_2a_0+a'_{i+j-1}b_1a_0+a'_{i+j}b_0a_0$. Since
$b'_{i+j-1}a_0$, $b'_{i+j-2}a_0$,...,$b'_1a_0$, $b'_0a_0$ are
central nilpotent by induction hypothesis and $b_{i+j}a_0$,
$b_{i+j-1}a_0$,..., $b_2a_0$, $b_1a_0$, $b_0a_0$ are central
nilpotent,  $a_0b'_{i+j}a_0$ and so $a_0b'_{i+j}$ and
$b'_{i+j}a_0$ are central nilpotent. Similarly,  multiplying $(i +
j)$ from the left by $b_0$,   we have $b_0a_0b'_{i+j}+$
$b_0a_1b'_{i+j-1}+$ $b_0a_2b'_{i+j-2}+ \ldots +b_0a_{i+j-1}b'_1+$
$b_0a_{i+j}b'_0+$ $b_0a'_0b_{i+j}+$ $b_0a'_1b_{i+j-1}+ \ldots +
b_0a'_{i+j-2}b_2+$ $b_0a'_{i+j-1}b_1+$ $b_0a'_{i+j}b_0=0$. Since
$b_0a_0$, $b_0a_1$, $b_0a_2$, ..., $b_0a_{i+j-1}$, $b_0a_{i+j}$,
$b_0a'_0$, $b_0a'_1$, ..., $b_0a'_{i+j-2}$, $b_0a'_{i+j-1}$ are
central nilpotent, $b_0a'_{i+j}b_0$ is central nilpotent and
therefore $b_0a'_{i+j}$ and $a'_{i+j}b_0$ are central nilpotent.
Proceeding in this manner, we finally obtain $a_kb'_l$, $b'_la_k$,
$a'_kb_l$, $b_la'_k$ are central nilpotent for all $k$ and $l$
with $k+l\leq i + j$. This completes the induction hypothesis.
Consequently $a_ib'_j$, $b'_ja_i$, $a'_ib_j$, $b_ja'_i$ are
central nilpotent for all $i$, $j$ with $i+j\leq n+t$. Then
\begin{center}$ \left [
\begin{array}{cc}a_i & a'_i\\0 & a_i \end{array}\right]\left
[\begin{array}{cc}b_j & b'_j\\0 & b_j \end{array}\right] = \left [
\begin{array}{cc}
a_ib_j & a_ib'_j+a'_ib_j\\
0 & a_ib_j \end{array}\right]$\end{center} is central in $T(R,R)$,
since from what we have proved that all entries of that matrix are
central in $R$. \hfill $\Box$\par\bigskip\end{proof}



\begin{footnotesize}

\end{footnotesize}

\begin{thebibliography}{99}

\bibitem{AHH} N. Agayev, A.  Harmanci and S. Halicioglu, {\it Extended Armendariz
Rings}, Algebras Groups Geom. 26(4)(2009), 343-354.

\bibitem{AGHH}  N. Agayev, G. Gungoroglu,  A.  Harmanci and S. Halicioglu,
{\it Central Armendariz Rings}, Bull. Malays. Math. Sci. Soc. (2)
34(1)(2011), 137-145.


\bibitem{AOH} N. Agayev, T. Ozen and   A. Harmanci, {\it On a Class of Semicommutative Rings},
Kyungpook Math. J. 51(2011), 283-291.




\bibitem{AC1} D. D.  Anderson and V. Camillo, {\it  Armendariz rings and Gaussian
rings}, Comm. Algebra 26(7)(1998), 2265-2272.

\bibitem{An} R. Antoine, {\it Nilpotent elements and Armendariz rings}, J. Algebra 319(2008), 3128-3140.

\bibitem{AE} E. Armendariz, {\it A note on extensions of Baer and p.p.-rings}, J. Austral.
Math. Soc. 18(1974), 470-473.

 \bibitem{BKP} G. F. Birkenmeier, J. Y. Kim and J. K.  Park, {\it On extensions of Baer and quasi-Baer Rings},
 J. Pure Appl. Algebra  159(2001), 25-42.


\bibitem{BK} G. F. Birkenmeier, J. Y. Kim and J. K. Park, {\it Principally quasi-Baer
rings}, Comm. Algebra  29(2)(2001), 639-660.

\bibitem{HC} H. Chen, Rings Related Stable Range Conditions,
Series in Algebra 11, World Scientific, Hackensack, NJ, 2011.






\bibitem{Hi}  Y. Hirano, {\it Some Studies of Strongly $\pi$-Regular Rings}, Math. J. Okayama Univ.
20(2)(1978), 141-149.


\bibitem{HKKwak} C. Y. Hong, N. K. Kim and T. K. Kwak, {\it On skew Armendariz
rings}, Comm. Algebra 31(1)(2003), 103-122.



\bibitem{HJP} S. U. Hwang, C. H. Jeon and K. S. Park, {\it A Generalization of Insertion of Factors Property}, Bull.
Korean Math. Soc.  44(1)(2007), 87-94.






\bibitem{KMS} D. Khurana, G. Marks and A. Srivastava, {\it  On unit-central
rings}, Advances in ring theory, 205-212, Trends Math.,
Birkhauser-Springer Basel AG, Basel, 2010.





\bibitem{LW} T. K.  Lee and T. L.  Wong, {\it On Armendariz rings}, Houston J. Math. 29(2003), 583-593.

\bibitem{LWL} L. Liang and L. Wang and Z. Liu, {\it On a generalization of semicommutative rings},
 Taiwanese J. Math. 11(5)(2007), 1359-1368.

\bibitem{LZh}L. Liu and R. Zhao, {\it  On weak Armendariz rings}, Comm. Algebra 34(7)(2006), 2607-2616.
\bibitem{RC} M. B. Rege and S. Chhawchharia, {\it Armendariz rings}, Proc. Japan Acad. Ser.A, Math. Sci. 73(1997), 14-17.
\bibitem{Sh} G. Shin, {\it Prime ideals and Sheaf Represantations of a Pseudo Symmetric ring}, Trans. Amer. Math. Soc. 184(1973),
43-69.
\end{thebibliography}
\end{document}